\documentclass[11pt]{amsart}
\usepackage[margin=1.15in]{geometry}
\usepackage{amsmath,amssymb,amsthm,mathtools}
\usepackage{enumitem}
\usepackage{microtype}
\usepackage[hidelinks]{hyperref}
\hypersetup{pdftitle={Weak (1,1) Schur multipliers and the H\"ormander--Mikhlin--Schur theorem},pdfauthor={Yongqiang Tian}}

\numberwithin{equation}{section}
\newtheorem{theorem}{Theorem}[section]
\newtheorem{proposition}[theorem]{Proposition}
\newtheorem{lemma}[theorem]{Lemma}
\newtheorem{corollary}[theorem]{Corollary}
\theoremstyle{definition}
\newtheorem{definition}[theorem]{Definition}
\theoremstyle{remark}
\newtheorem{remark}[theorem]{Remark}

\newcommand{\RR}{\mathbb R}
\newcommand{\C}{\mathbb C}
\newcommand{\Z}{\mathbb Z}

\newcommand{\Tr}{\operatorname{Tr}}
\newcommand{\rank}{\operatorname{rank}}
\newcommand{\dist}{\operatorname{dist}}
\newcommand{\ran}{\operatorname{ran}}
\newcommand{\norm}[1]{\lVert#1\rVert}
\newcommand{\ip}[2]{\langle #1,#2\rangle}
\newcommand{\cB}{\mathcal B}
\newcommand{\cK}{\mathcal K}
\newcommand{\Sone}{\mathcal S^1}
\newcommand{\Stwo}{\mathcal S^2}
\newcommand{\Sw}{\mathcal S^{1,\infty}}
\newcommand{\Wa}{W_{\mathrm{amp}}}
\newcommand{\cD}{\mathcal D}
\newcommand{\cH}{\mathcal H}
\newcommand{\fD}{\mathfrak D}
\newcommand{\maxnorm}[1]{\norm{#1}_\infty}

\begin{document}

\title[Weak $(1,1)$ Schur multipliers and the HMS theorem]{Weak $(1,1)$ Schur multipliers and the H\"ormander--Mikhlin--Schur theorem}

\author{Yongqiang Tian}
\address{School of Mathematics and Statistics, Central South University, Changsha, Hunan 410083, P. R. China}
\email{tianyongqiang@csu.edu.cn}


\begin{abstract}
Let $S_m$ be the Schur multiplier with symbol $m$. We characterize the boundedness of $S_m:\mathcal S^1\to\mathcal S^{1,\infty}$ by approximation of weighted columns in a common finite-dimensional subspace. If $W(m)$ is the weak $(1,1)$ bound and $\fD(m)$ is the column approximation constant, then $\tfrac12W(m)\le\fD(m)\le2W(m)$. These estimates hold uniformly under amplification. We also prove the H\"ormander--Mikhlin--Schur weak $(1,1)$ estimate in every Euclidean dimension under the usual integer-order condition. The proof combines Cuculescu's projections, adjacent dyadic grids and tensor polynomial interpolation on a binary block decomposition.
\end{abstract}

\maketitle

\section{Introduction}

\subsection{Notation}
Throughout, Hilbert spaces are complex and separable, and inner products are linear in the first variable. For a vector $\xi$ in a Hilbert space $\cH$, we write $\norm{\xi}_{\cH}=\ip{\xi}{\xi}_{\cH}^{1/2}$ and omit the subscript when the space is clear. In particular, vector norms and distances in $\C^F$ are taken in $\ell^2(F)$. We write $\cB(\cH)$ and $\cK(\cH)$ for the bounded and compact operators on $\cH$, respectively. For $X\in\cK(\cH)$, let $s_1(X)\ge s_2(X)\ge\cdots\ge0$ be its singular values, counted with multiplicity and extended by zeros. We write $\Tr$ for the canonical unnormalized trace and $\norm X_\infty$ for the operator norm of $X$. Put
\[
 \norm X_p=\Tr(|X|^p)^{1/p}=\Bigl(\sum_{n\ge1}s_n(X)^p\Bigr)^{1/p},\qquad
 \norm X_{1,\infty}=\sup_{n\ge1}n\,s_n(X).
\]
The corresponding ideals are the Schatten--von Neumann class $\mathcal S^p$ ($1\le p<\infty$) and the weak trace class $\Sw$. We set $\mathcal S^\infty=\cK(\cH)$ with the operator norm, and $\mathbb N_0=\{0,1,\ldots\}$. If we set
\[
 N(t;X)=\#\{n:s_n(X)>t\}\qquad(t>0),
\]
then $\norm X_{1,\infty}=\sup_{t>0}tN(t;X)$. We shall use the standard inequalities
\begin{align}
 s_n(X)&=\min_{\rank(R)<n}\norm{X-R}_\infty,\label{eq:approxnumbers}\\
 s_{i+j-1}(X+Y)&\le s_i(X)+s_j(Y),\qquad
 s_n(UXV)\le\norm U_\infty\norm V_\infty s_n(X).\label{eq:weyl}
\end{align}
These formulas also hold for rectangular operators; see \cite[Chapter II]{GK69} and \cite[Chapter 11]{BS87}. Constants denoted by $C$ are absolute, and subscripts indicate their permitted dependence; their values may change from line to line.

Fix a countable index set $I$. We identify operators on $\ell^2(I)$ with their matrices in the standard basis, and call a matrix \emph{finitely supported} if it has only finitely many nonzero entries. For a scalar symbol $m=(m_{ij})_{i,j\in I}$, define
\[
 S_m(X)=m\circ X=(m_{ij}x_{ij})_{i,j\in I}
\]
on finitely supported matrices, and set
\begin{equation}\label{eq:W}
 W(m)=\sup\{\norm{S_m(X)}_{1,\infty}:X\text{ finitely supported},\ \norm X_1\le1\}.
\end{equation}
Thus $W(m)<\infty$ is precisely the condition that $S_m$ extend boundedly from $\Sone$ to $\Sw$; see Lemma~\ref{lem:extension}. For a scalar symbol, $\maxnorm m=\sup_{i,j}|m_{ij}|$ denotes the supremum of its entries in absolute value.

For $d\ge1$, write $M_d=\cB(\C^d)$ and let $\mathbf J_d=(1)_{\alpha,\beta=1}^d$. The symbol $m\otimes\mathbf J_d$ on $I\times\{1,\ldots,d\}$ has entries
$$(m\otimes\mathbf J_d)_{(i,\alpha),(j,\beta)}=m_{ij}.$$
Thus
$$S_{m\otimes\mathbf J_d}=S_m\otimes\mathrm{id}_{M_d},\qquad
 \Wa(m)=\sup_{d\ge1}W(m\otimes\mathbf J_d).$$

\begin{remark}\label{rem:rectangular}
A symbol on $I_1\times I_2$ can be regarded as a square symbol on the disjoint union $I_1\sqcup I_2$ by setting
\[
 \widetilde m=\begin{pmatrix}0&m\\0&0\end{pmatrix}.
\]
Compression to a corner is contractive on $\Sone$, and a rectangular matrix has the same nonzero singular values as its embedding in this corner. Hence $W(\widetilde m)=W(m)$ and $\Wa(\widetilde m)=\Wa(m)$. It therefore suffices to use one index set $I$. We retain separate finite row and column sets $F,G\subset I$ in the formulas below.
\end{remark}

For a matrix $c$ on $F\times G$, define its Hilbert-space factorization norm by
\begin{equation}\label{gamma2}
 \gamma_2(c)=\inf\Bigl\{\max_i\norm{u_i}_{\cH}\max_j\norm{v_j}_{\cH}:
 u_i,v_j\in \cH,\ c_{ij}=\ip{u_i}{v_j}_{\cH}\Bigr\},
\end{equation}
where the infimum runs over common Hilbert spaces $\cH$ and all such factorizations. For an infinite symbol, take the supremum over finite restrictions. By the Schur multiplier factorization theorem and trace duality, we have
\[
 \gamma_2(c)=\norm{S_c:\Sone\to\Sone}
 =\sup_{d\ge1}\norm{S_{c\otimes\mathbf J_d}:\Sone\to\Sone};
\]
see \cite[Theorem 3.2 and Corollary 3.3]{PPS89} and \cite{Pis01}. If $c$ is finite and $\gamma_2(c)\le1$, then the same theorem gives a factorization with $\norm{u_i}_{\cH},\norm{v_j}_{\cH}\le1$.

\subsection{Background and main results}

For $1<p<\infty$, Conde-Alonso, Gonz\'alez-P\'erez, Parcet and Tablate proved a H\"ormander--Mikhlin condition for boundedness on $\mathcal S^p$ \cite[Theorem A]{CGPT23}. A classical endpoint example is triangular truncation, which is unbounded on $\Sone$ but bounded from $\Sone$ to $\Sw$ \cite{Mac61,GK70}.

Parcet's ICM 2026 survey \cite[Remark 1.3, p.~733]{Par26} records the weak $(1,1)$ extension of the H\"ormander--Mikhlin--Schur theorem as an open problem. Theorem~\ref{thm:HMSn} and Corollary~\ref{cor:continuousn} below establish its usual integer-order formulation in every Euclidean dimension, with bounds uniform under amplification.

We also characterize weak $(1,1)$ Schur multipliers by approximation of weighted columns in a common finite-dimensional subspace. The endpoint proof is independent of this characterization. It combines block decompositions indexed by finite rooted trees with Cuculescu's projections \cite{Cuc71} and the weak-type martingale method of \cite{Ran02}. Related tree filtrations for Schur multipliers appear in \cite[Section 1.1]{GPPW26}.

\subsection{Characterization of weak $(1,1)$ Schur multipliers}
For finite nonempty $F,G\subset I$, a matrix $c$ on $F\times G$ with $\gamma_2(c)\le1$, and $p\in[0,\infty)^F$ with $\norm p_2=1$, put
\begin{equation}\label{eq:columns}
 z_j=(p_i m_{ij}c_{ij})_{i\in F}\in\ell^2(F)\qquad(j\in G).
\end{equation}
By Lemma~\ref{lem:factor}, every matrix on $F\times G$ of trace norm one has the form $X=(p_i c_{ij}q_j)_{i,j}$, where $\gamma_2(c)\le1$ and $p,q$ are nonnegative unit vectors in $\ell^2(F),\ell^2(G)$. Therefore its image under $S_m$ has columns $q_jz_j$.

For a finite family $z=(z_j)_{j\in G}$, set
\[
 d_k(z)=\inf_{\dim E\le k}\max_{j\in G}\dist(z_j,E),\qquad k\in\mathbb N_0,
\]
where $E\subset\C^F$ is a subspace. This is the Kolmogorov $k$-width of $\{z_j:j\in G\}$; see \cite{Pin85}. Define
\begin{equation}\label{eq:width}
 \fD(m)=\sup_{F,G,c,p}\sup_{k\in\mathbb N_0}\sqrt{k+1}\,d_k(z).
\end{equation}
Here $F,G,c,p$ range over the choices above. We also set $\fD_{\mathrm{amp}}(m)=\sup_{d\ge1}\fD(m\otimes\mathbf J_d)$.

\begin{theorem}\label{thm:char}
For every scalar symbol $m$, the Schur multiplier $S_m$ extends boundedly from $\Sone$ to $\Sw$ if and only if $\fD(m)<\infty$. Moreover,
\begin{equation}\label{eq:comparisons}
 \tfrac12W(m)\le\fD(m)\le2W(m).
\end{equation}
The same inequalities hold for $W_{\mathrm{amp}}$ and $\fD_{\mathrm{amp}}$. Thus $\fD_{\mathrm{amp}}(m)<\infty$ is equivalent to uniform weak $(1,1)$ boundedness under amplification.
\end{theorem}

\subsection{Block decompositions indexed by a tree}
Let $\Omega$ be a finite nonempty index set. Starting with $\Omega_o=\Omega$, split each nonterminal block into two nonempty disjoint blocks,
\[
 \Omega_v=\Omega_{v_-}\sqcup\Omega_{v_+}.
\]
These successive partitions are indexed by a finite rooted full binary tree $\mathcal T$: the root is $o$, each internal vertex $v$ has children $v_-,v_+$, and the leaves index the terminal blocks. We write $\mathcal T^\circ$ and $\mathcal L(\mathcal T)$ for the internal vertices and leaves, respectively; see \cite[Section 1.5]{Die17} for related terminology. Put $N_\lambda=\Omega_\lambda$ for a leaf $\lambda$. The terminal blocks may have arbitrary finite cardinalities, and the root may itself be terminal.

On $\cH=\ell^2(\Omega)$, let $P_u$ be the coordinate projection onto $\ell^2(\Omega_u)$. Thus
\[
 P_o=I,\qquad P_v=P_{v_-}+P_{v_+},\qquad P_{v_-}P_{v_+}=0.
\]
Write $u\preceq v$ when $u$ is an ancestor of $v$, equivalently $P_v\le P_u$, and $u\prec v$ when $u\preceq v$ and $u\ne v$. Depth counts the number of splits from the root. Two distinct terminal blocks separate at a unique split; its vertex is their \emph{least common ancestor}. The corresponding entries belong to the corners $P_{v_-}\cB(\cH)P_{v_+}$ and $P_{v_+}\cB(\cH)P_{v_-}$.

For $r\in\mathbb N_0$, let $\mathcal P_r(u)$ index the partition obtained from $\Omega_u$ after $r$ refinements, splitting every nonterminal block at each step and retaining terminal blocks. In particular, $\mathcal P_0(u)=\{u\}$, and
\[
 P_u=\sum_{w\in\mathcal P_r(u)}P_w,\qquad
 P_wP_{w'}=0\quad(w\ne w'),\qquad |\mathcal P_r(u)|\le2^r.
\]

\begin{definition}\label{def:complexity}
A symbol $a$ on $\Omega\times\Omega$ has \emph{descendant complexity at most $r$} if $S_a$ acts by scalar multiplication on each terminal corner $P_\lambda\cB(\cH)P_\lambda$ and on each corner
\[
 P_w\cB(\cH)P_{w'},\qquad P_{w'}\cB(\cH)P_w,
 \quad w\in\mathcal P_r(v_-),\quad w'\in\mathcal P_r(v_+).
\]
Here $\lambda$ ranges over the leaves and $v$ over the internal vertices. Equivalently, $a$ is constant on the corresponding matrix rectangles. The scalar coefficients on different corners, including opposite corners, may differ.
\end{definition}

For $r=0$, each off-diagonal corner at a split has one scalar coefficient. For $r=1$, each such corner is divided into at most four corners. The parameter $r$ counts refinements below each split, independently of the depth of the whole tree.

\begin{theorem}\label{thm:tree}
If $a$ has descendant complexity at most $r$ on a finite rooted full binary tree, then
\begin{equation}\label{eq:FC}
 \Wa(a)\le6\cdot2^{r/2}\maxnorm a.
\end{equation}
The bound is independent of the tree, the dimensions of the terminal blocks, and the amplification. For every $r\in\mathbb N_0$, there is such a symbol with $\maxnorm a=1$ and $W(a)=\Wa(a)=2^{(r+1)/2}$.
\end{theorem}

\subsection{The HMS endpoint}
Let $n\ge1$, let $s\ge1$ be an integer, and let
$M\in C^s(\RR^{2n}\setminus\{x=y\})$. For a multi-index $\alpha\in\mathbb N_0^n$, write $|\alpha|=\alpha_1+\cdots+\alpha_n$ and $x^\alpha=x_1^{\alpha_1}\cdots x_n^{\alpha_n}$. With $|x-y|$ denoting Euclidean distance, put
\begin{equation}\label{eq:HMSnorm}
 \begin{aligned}
 K_0&=\sup_{x\ne y}|M(x,y)|,\\
 H_{n,s}(M)&=K_0+
 \sum_{1\le|\alpha|\le s}\sup_{x\ne y}|x-y|^{|\alpha|}
 \bigl(|\partial_x^\alpha M(x,y)|+|\partial_y^\alpha M(x,y)|\bigr).
 \end{aligned}
\end{equation}
For bounded $M$ and $X\in\Stwo(L_2(\RR^n))$ with kernel $k_X$, let $S_M(X)$ be the operator with kernel $M(x,y)k_X(x,y)$. If $H_{n,s}(M)<\infty$ and $1<p<\infty$, then the integer-order H\"ormander--Mikhlin--Schur theorem \cite[Theorem A]{CGPT23} gives
\[
 \sup_{d\ge1}\norm{S_M\otimes\mathrm{id}_{M_d}:\mathcal S^p\to\mathcal S^p}
 \le C_n\frac{p^2}{p-1}H_{n,s}(M),
 \qquad s=\lfloor n/2\rfloor+1.
\]

For atomic restrictions we set $M(x,x)=0$, including all blocks of repeated points. Any bounded choice of the values $M(x,x)$ may be added afterwards, since the resulting block-diagonal Schur multiplier has trace-class norm at most $\sup_x|M(x,x)|$, also under amplification.

\begin{theorem}\label{thm:HMSn}
Suppose $s>n/2$ and $H_{n,s}(M)<\infty$. Then there is a constant $C_{n,s}$ such that all finite lists $(x_i)$ and $(y_j)$ in $\RR^n$, with repetitions allowed, satisfy
\begin{equation}\label{eq:endpointn}
 \Wa\bigl((M(x_i,y_j))_{i,j}\bigr)
 \le C_{n,s}K_0^{\,1-n/(2s)}H_{n,s}(M)^{\,n/(2s)}.
\end{equation}
If $K_0=0$, the right-hand side is understood as zero. In particular, the choice $s=\lfloor n/2\rfloor+1$ gives the bound $C_nH_{n,s}(M)$ under the usual integer-order condition.
\end{theorem}

\begin{corollary}\label{cor:continuousn}
Under the hypotheses of Theorem~\ref{thm:HMSn},
\[
 \sup_{d\ge1}\norm{S_M\otimes\mathrm{id}_{M_d}:
 \Sone(L_2(\RR^n;\C^d))\to\Sw(L_2(\RR^n;\C^d))}
 \le C_{n,s}K_0^{\,1-n/(2s)}H_{n,s}(M)^{\,n/(2s)}.
\]
\end{corollary}

\paragraph{\bf Organization of the paper.}
Section~\ref{sec:prelim} contains preliminary lemmas. Section~\ref{sec:char} proves the characterization, and Section~\ref{sec:tree} proves the tree estimate. Sections~\ref{sec:grids} and~\ref{sec:trees} give the localization and endpoint argument in one dimension. Section~\ref{sec:productgrids} constructs the smooth localization on product grids. Finally, Section~\ref{sec:hdHMS} combines polynomial interpolation with the tree estimate to prove Theorem~\ref{thm:HMSn} and Corollary~\ref{cor:continuousn} in every dimension.

\section{Preliminaries}\label{sec:prelim}

\begin{lemma}[Factorization]\label{lem:factor}
Let $F,G$ be finite nonempty index sets. Every matrix $X$ on $F\times G$ with $\norm X_1=1$ can be written as
\[
 X=(p_i c_{ij}q_j)_{i\in F,j\in G},
 \qquad p_i,q_j\ge0,\quad \norm p_2=\norm q_2=1,\quad\gamma_2(c)\le1.
\]
Conversely, every matrix of this form has trace norm at most one.
\end{lemma}

\begin{proof}
Take a reduced singular value decomposition $X=UTV^*$ and set $A=UT^{1/2}$, $B=VT^{1/2}$. Then $X=AB^*$ and $\norm A_2^2=\norm B_2^2=\Tr(T)=1$. If $\alpha_i,\beta_j$ are the rows of $A,B$, set $p_i=\norm{\alpha_i}$ and $q_j=\norm{\beta_j}$. Normalize each nonzero row to a unit vector, and replace each zero row by the zero vector. Their pairwise inner products therefore define $c$ with $\gamma_2(c)\le1$ and give the required factorization, with $\sum_i p_i^2=\sum_j q_j^2=1$.

Conversely, choose $c_{ij}=\ip{u_i}{v_j}$ with $\norm{u_i},\norm{v_j}\le1$. Form matrices $A,B$ with rows $p_i u_i$ and $q_j v_j$. Then $(p_i c_{ij}q_j)_{i,j}=AB^*$ and
\[
 \norm{AB^*}_1\le\norm A_2\norm B_2
 \le\Bigl(\sum_i p_i^2\Bigr)^{1/2}\Bigl(\sum_j q_j^2\Bigr)^{1/2}=1.
\]
\end{proof}

For a finite matrix $X$ and $\tau>0$, put
\begin{equation}\label{eq:Phi}
\Lambda_{\tau}(X)=\sum_{n\ge1}\min\{s_n(X)^2,\tau^2\}.
\end{equation}

\begin{lemma}\label{lem:Phi}
For every finite matrix $X$, with $P$ acting on its codomain, we have
\begin{align}
\Lambda_{\tau}(X)&=\int_0^{\tau^2}N(\sqrt t;X)\,dt=\min_{0\le P\le I}\{\tau^2\Tr(P)+\Tr(XX^*(I-P))\},\label{eq:Phi-var}\\
\Lambda_{\tau}(X)&\le2\tau\norm X_{1,\infty}.\label{eq:Phi-weak}
\end{align}
The minimum in \eqref{eq:Phi-var} is attained at $P=\mathbf1_{(\tau^2,\infty)}(XX^\ast)$.
\end{lemma}

\begin{proof}
The integral formula follows by summing $\min\{s^2,\tau^2\}=\int_0^{\tau^2}\mathbf1_{\{t<s^2\}}\,dt$ over the singular values. In an eigenbasis of $XX^*$, the variational expression is $\sum_n[s_n^2+(\tau^2-s_n^2)P_{nn}]$. Since $0\le P_{nn}\le1$, it is minimized by the spectral projection $P=\mathbf1_{(\tau^2,\infty)}(XX^\ast)$. 

Finally, we integrate $$N(\sqrt t;X)\le\norm X_{1,\infty}/\sqrt t$$ over $(0,\tau^2)$. Together with \eqref{eq:Phi-var}, this gives \eqref{eq:Phi-weak}.
\end{proof}

The singular-value inequalities \eqref{eq:approxnumbers}--\eqref{eq:weyl} imply the following standard estimates
$$N(t;X)\le\rank(R)+N(t;X-R),\qquad N(t;X)\le t^{-2}\norm X_2^2,$$
and, for finitely many compact operators $X_j$ and numbers $t_j>0$,
\begin{equation}\label{eq:countsum}
 N\Bigl(\sum_jt_j;\sum_jX_j\Bigr)\le\sum_jN(t_j;X_j).
\end{equation}
See \cite[Chapter II]{GK69} and \cite[Chapter 11]{BS87}.

\begin{lemma}\label{lem:weaksum}
For $X_1,\ldots,X_n\in\Sw$,
\begin{equation}\label{eq:weaksum}
 \bigl\|\sum_{j=1}^nX_j\bigr\|_{1,\infty}
 \le\Bigl(\sum_{j=1}^n\norm{X_j}_{1,\infty}^{1/2}\Bigr)^2.
\end{equation}
\end{lemma}

\begin{proof}
The assertion is immediate if all $X_j$ vanish. Otherwise, discard zero summands and put $a_j=\norm{X_j}_{1,\infty}^{1/2}$ and $a=\sum_ja_j$. Applying \eqref{eq:countsum} with $t_j=ta_j/a$ then yields
$$N\Bigl(t;\sum_jX_j\Bigr)\le\sum_j\frac{a_j^2}{t_j}=\frac{a^2}{t},\qquad t>0.$$
The assertion follows immediately.
\end{proof}

\begin{lemma}\label{lem:extension}
If $W(m)<\infty$, then $\maxnorm m\le W(m)$ and the Schur multiplier $S_m:\Sone\to\Sw$ is bounded with $$\|S_m:\mathcal{S}^1\to \mathcal{S}^{1,\infty}\|\leq W(m).$$
\end{lemma}

\begin{proof}
Testing a matrix unit gives $|m_{ij}|\le W(m)$. Thus $S_m$ defines a bounded map on $\Stwo$. Fix $X\in\Sone$. Then
\begin{equation}\label{eq:12}
\norm{S_m(X)}_2\le\maxnorm m\,\norm X_2\le W(m)\norm X_1.
\end{equation}
Let $P_n$ be finite-rank coordinate projections such that $P_n\uparrow I$, and put $X_n=P_nXP_n$. Then $X_n\to X$ in $\mathcal{S}^1$ and $\norm{X_n}_1\le\norm X_1$. By \eqref{eq:12}, we have $S_m(X_n)\to S_m(X)$ in $\mathcal{S}^2$. For every $k$, the map $Y\to s_k(Y)$ is continuous in operator norm. Therefore \eqref{eq:W} gives
 $$s_k(S_m(X))=\lim_{n\to\infty} s_k(S_m(X_n))\le\limsup_{n\to\infty}\frac{W(m)}{k}\norm{P_nXP_n}_1\le\frac{W(m)}{k}\norm X_1,\qquad k\geq1.$$
The proof is complete.
\end{proof}

\section{Characterization of weak $(1,1)$ Schur multipliers}\label{sec:char}

Fix a scalar symbol $m=(m_{ij})_{i,j\in I}$ and choose $F,G,c,p$ and the columns $z_j$ as in \eqref{eq:columns}. For $q\in[0,\infty)^G$ with $\norm q_2=1$, put
\begin{equation}\label{eq:Zq}
 Z_q=(p_i m_{ij}c_{ij}q_j)_{i\in F,j\in G},\qquad
 \Gamma_q=Z_qZ_q^*=\sum_jq_j^2z_jz_j^*.
\end{equation}
Lemma~\ref{lem:factor} gives $\norm{(p_i c_{ij}q_j)_{i,j}}_1\le1$, so $\norm{Z_q}_{1,\infty}\le W(m)$. The matrices $\Gamma_q$ range over the compact convex set
\[
 \mathcal C=\operatorname{conv}\{z_jz_j^*:j\in G\}.
\]

\begin{lemma}\label{lem:minimax}
For every $\tau>0$, there is a positive contraction $P\in\cB(\C^F)$ such that
\begin{equation}\label{eq:common-contraction}
 \tau^2\Tr(P)+\max_j\ip{(I-P)z_j}{z_j}\le2\tau W(m).
\end{equation}
\end{lemma}

\begin{proof}
For $0\le P\le I$,
\[
 \max_j\ip{(I-P)z_j}{z_j}
 =\max_{\Gamma\in\mathcal C}\Tr(\Gamma(I-P)).
\]
The function $\tau^2\Tr(P)+\Tr(\Gamma(I-P))$ is continuous and affine in $P$ and $\Gamma$, and both domains are compact and convex. By Sion's minimax theorem \cite{Sio58} and Lemma~\ref{lem:Phi},
\[
 \begin{aligned}
 &\min_{0\le P\le I}\bigl\{\tau^2\Tr(P)+\max_j\ip{(I-P)z_j}{z_j}\bigr\}\\
 &\quad=\max_{\Gamma\in\mathcal C}\min_{0\le P\le I}
 \bigl\{\tau^2\Tr(P)+\Tr(\Gamma(I-P))\bigr\}\\
 &\quad=\max_{q\ge0,\,\norm q_2=1}\Lambda_\tau(Z_q)
 \le2\tau W(m).
 \end{aligned}
\]
Since the minimum is attained, the assertion follows.
\end{proof}

\begin{proof}[Proof of Theorem~\ref{thm:char}]
\emph{Necessity.} Assume $W(m)<\infty$ and fix $L>2W(m)$. Let $F,G,c,p$ and $z_j$ be as above, put $n=|F|$, and fix $k\in\mathbb N_0$. If $k\ge n$, then $d_k(z)=0$. Otherwise, set $\tau=L/(k+1)$ and choose $P$ as in Lemma~\ref{lem:minimax}. Write
\[
 t=\Tr(P),\qquad e=\max_j\ip{(I-P)z_j}{z_j}.
\]
Then $t\le2W(m)/\tau<k+1$ and $e\le L\tau-\tau^2t$. Let $\mu_1\ge\cdots\ge\mu_n$ be the eigenvalues of $P$, and let $Q$ project onto its first $k$ eigenvectors, with $Q=0$ if $k=0$. Since $\mu_{k+1}\le t/(k+1)<1$,
\[
 I-Q\le\frac{I-P}{1-t/(k+1)}.
\]
Thus, for $E=\ran Q$,
\[
 \max_j\dist(z_j,E)^2
 \le\frac{L\tau-\tau^2t}{1-t/(k+1)}
 =\frac{L^2}{k+1}.
\]
Taking the suprema in \eqref{eq:width} and letting $L\downarrow2W(m)$ gives $\fD(m)\le2W(m)$.

\emph{Sufficiency.} Assume $\fD(m)<\infty$ and fix $A>\fD(m)$. Apply Lemma~\ref{lem:factor} to a finitely supported $X$ with $\norm X_1=1$, and choose $F,G,c,p,q$ accordingly. Then $Y=S_m(X)=Z_q$. For $n\ge1$, put $k=\lfloor n/2\rfloor$ and choose a subspace $E$ of dimension at most $k$ with $\max_j\dist(z_j,E)\le A/\sqrt{k+1}$. If $Q$ is its orthogonal projection and $R=QY$, then
\[
 \rank(R)\le k,\qquad
 \norm{Y-R}_2^2=\sum_jq_j^2\norm{(I-Q)z_j}^2\le\frac{A^2}{k+1}.
\]
Therefore, by the rank perturbation inequality and the Hilbert--Schmidt bound, we have
\[
 s_n(Y)\le s_{n-k}(Y-R)
 \le\frac{A}{\sqrt{(k+1)(n-k)}}
 =\frac{A}{\sqrt{\lfloor(n+1)^2/4\rfloor}}\le\frac{2A}{n}.
\]
Taking the supremum over $n,X$ and letting $A\downarrow\fD(m)$ gives $W(m)\le2\fD(m)$.

Finally, Lemma~\ref{lem:extension} gives the bounded extension. We then apply both inequalities to $m\otimes\mathbf J_d$ and take the supremum over $d$ to obtain the amplified statement.
\end{proof}

\section{The estimate for block decompositions}\label{sec:tree}

Fix a finite nonempty index set $\Omega$ and a block decomposition indexed by a finite rooted full binary tree $\mathcal T$. On $\cH=\ell^2(\Omega)$, let $P_u$ be the coordinate projection onto $\ell^2(\Omega_u)$. For $r\in\mathbb N_0$, let $a=(a_{ij})_{i,j\in\Omega}$ be a scalar Schur multiplier symbol with descendant complexity at most $r$ in the sense of Definition~\ref{def:complexity}; thus $S_a(X)=(a_{ij}x_{ij})_{i,j\in\Omega}$. Put $A=\maxnorm a$ and assume $A>0$. Amplification replaces each terminal space $\ell^2(N_\lambda)$ by $\ell^2(N_\lambda)\otimes\C^d$, preserving the block decomposition, the complexity and $A$. It therefore suffices to estimate $W(a)$ uniformly over the dimensions of the terminal spaces.

Starting with the terminal partition $\pi_0=\{N_\lambda:\lambda\in\mathcal L(\mathcal T)\}$, successively merge two child blocks into their parent to obtain partitions $\pi_0,\ldots,\pi_N$, where $N=|\mathcal T^\circ|$ and $\pi_N=\{\Omega\}$. Put
\[
 \mathcal M_k=\bigoplus_{B\in\pi_k}\cB(\ell^2(B)),\qquad
 E_k(X)=\sum_{B\in\pi_k}P_BXP_B,
\]
where $P_B$ is the coordinate projection onto $\ell^2(B)$. The algebras $\mathcal M_k$ form an increasing filtration ending at $\cB(\cH)$, and $E_k$ is the trace-preserving conditional expectation onto $\mathcal M_k$. Thus $(E_k(X))_{k=0}^N$ is a martingale. If the $k$-th merge joins the children of $v$, its martingale difference is
\[
 d_kX:=(E_k-E_{k-1})(X)
 =P_{v_-}XP_{v_+}+P_{v_+}XP_{v_-},\qquad 1\le k\le N.
\]
The martingale decomposition $X=E_0(X)+\sum_{k=1}^N d_kX$ therefore takes the form
\begin{equation}\label{eq:treeblocks}
 X=\sum_{\lambda\in\mathcal L(\mathcal T)}P_\lambda XP_\lambda
 +\sum_{v\in\mathcal T^\circ}
 (P_{v_-}XP_{v_+}+P_{v_+}XP_{v_-}).
\end{equation}
These matrix corners are mutually orthogonal for the Hilbert--Schmidt inner product: each matrix entry belongs either to one terminal diagonal block or to the unique split at which its two terminal blocks separate. A related use of tree partitions and martingale differences for Schur multipliers appears in \cite[Section 1.1]{GPPW26}.

\subsection{Rank estimates}

\begin{lemma}\label{lem:profiles}
Let $u\in\mathcal T$, let $e\le P_u$ be an orthogonal projection, and let $Y\in\cB(\cH)$. Then
\begin{equation}\label{eq:rankinflation}
 \begin{aligned}
 \rank\bigl(S_a(eY(I-P_u))\bigr)&\le2^r\rank(e),\\
 \rank\bigl(S_a((I-P_u)Ye)\bigr)&\le2^r\rank(e).
 \end{aligned}
\end{equation}
If $u$ is a leaf, then
\[
 \rank\bigl(S_a(eY(I-P_u)+eYe)\bigr)\le\rank(e).
\]
\end{lemma}

\begin{proof}
Fix $w\in\mathcal P_r(u)$ and a column index $j\notin\Omega_u$. Let $v$ be the unique split separating $\Omega_u$ from the terminal block containing $j$. On the side containing $\Omega_u$, the block $\Omega_w$ lies inside one block of $\mathcal P_r(v_\pm)$: refinement starts at $u$, at or below $v_\pm$, and continues for $r$ steps unless a terminal block is reached. The index $j$ lies in one block on the other side. Definition~\ref{def:complexity} therefore makes $a_{ij}$ independent of $i\in\Omega_w$. Thus the rows indexed by $\Omega_w$ agree outside $\Omega_u$. The same argument applies to columns.

Choose $i_w\in\Omega_w$, and let $B_w$ be the diagonal operator with entries $a_{i_wj}$ for $j\notin\Omega_u$ and zero otherwise. The preceding observation gives
\[
 S_a(eY(I-P_u))
 =\sum_{w\in\mathcal P_r(u)}P_weY(I-P_u)B_w.
\]
Since there are at most $2^r$ terms, each of rank at most $\rank(e)$, the first inequality follows. The column version is proved in the same way. If $u$ is a leaf, the rows indexed by $N_u$ agree outside $N_u$, and $a$ has a constant value $\zeta_u$ on $N_u\times N_u$. Hence
\[
 S_a(eY(I-P_u)+eYe)
 =eY\bigl((I-P_u)B_u+\zeta_ue\bigr),
\]
whose range lies in $e\cH$.
\end{proof}

\subsection{Cuculescu's projections}
For the next three lemmas, fix $X\ge0$ and $\kappa>0$. The martingale $(E_k(X))_{k=0}^N$ is positive, so we may apply Cuculescu's construction \cite{Cuc71}, in the form of \cite[Proposition 2.4]{Ran02}, one block at a time. Starting at the terminal blocks, define
\begin{equation}\label{eq:cuculescu}
 \begin{gathered}
 R_v=\begin{cases}P_v,&v\in\mathcal L(\mathcal T),\\Q_{v_-}+Q_{v_+},&v\in\mathcal T^\circ,\end{cases}
 \qquad X_v=R_vXR_v,\\
 p_v=\mathbf1_{(\kappa,\infty)}(X_v),\qquad Q_v=R_v-p_v.
 \end{gathered}
\end{equation}
The spectral projections are taken in $\cB(R_v\cH)$.

\begin{lemma}\label{lem:cuculescu}
The projections in \eqref{eq:cuculescu} satisfy:
\begin{enumerate}[label=\textup{(\roman*)},leftmargin=*]
\item For every $w$,
\[
 P_w=Q_w+\sum_{u:w\preceq u}p_u,
\]
and this sum is orthogonal. In particular, $Q_o$ and all $p_u$ are mutually orthogonal and sum to $I$.
\item $Q_vXQ_v\le\kappa Q_v$, $Q_vXp_v=0$, and $p_vXp_v\ge\kappa p_v$. Consequently,
\begin{equation}\label{eq:rankbudget}
 \kappa\sum_{u\in\mathcal T}\rank(p_u)\le\Tr(X).
\end{equation}
\item For $v\in\mathcal T^\circ$, $X_v\le2\kappa R_v$, and hence $\norm{p_vXp_v}_\infty\le2\kappa$.
\end{enumerate}
\end{lemma}

\begin{proof}
Part (i) follows by iterating $R_v=Q_v+p_v$ through the tree. Part (ii) consists of the standard Cuculescu estimates; see \cite[Proposition 2.4]{Ran02}. For (iii), write $S_v=Q_{v_-}-Q_{v_+}$. Positivity of $S_vXS_v$ and the child bounds in (ii) give
\[
 X_v\le X_v+S_vXS_v
 =2(Q_{v_-}XQ_{v_-}+Q_{v_+}XQ_{v_+})\le2\kappa R_v.
\]
\end{proof}

For leaves, however, the bound $\norm{p_\lambda Xp_\lambda}_\infty\le2\kappa$ may fail. We therefore omit these terms from $X_\flat$ below.

\subsection{A Hilbert--Schmidt estimate}
Define
\begin{equation}\label{eq:retained}
 X_\flat=\sum_{\lambda\in\mathcal L(\mathcal T)}Q_\lambda XQ_\lambda
 +\sum_{v\in\mathcal T^\circ}
 (Q_{v_-}XQ_{v_+}+Q_{v_+}XQ_{v_-}).
\end{equation}

\begin{lemma}\label{lem:retained}
We have
\[
 \begin{gathered}
 \norm{S_a(X_\flat)}_2^2\le A^2\norm{X_\flat}_2^2,\\
 \norm{X_\flat}_2^2+2\kappa^2\sum_{u\in\mathcal T}\rank(p_u)
 \le3\kappa\Tr(X).
 \end{gathered}
\]
\end{lemma}

\begin{proof}
The summands of \eqref{eq:retained} lie in the mutually orthogonal matrix corners of \eqref{eq:treeblocks}, so their squared Hilbert--Schmidt norms add. At an internal vertex, the block decomposition relative to $Q_{v_-}+Q_{v_+}$ gives
\[
 \begin{aligned}
 &\norm{Q_{v_-}XQ_{v_+}}_2^2+\norm{Q_{v_+}XQ_{v_-}}_2^2\\
 &\quad=\norm{Q_vXQ_v}_2^2+\norm{p_vXp_v}_2^2
 -\norm{Q_{v_-}XQ_{v_-}}_2^2-\norm{Q_{v_+}XQ_{v_+}}_2^2.
 \end{aligned}
\]
Here $X_v=Q_vXQ_v+p_vXp_v$ by spectral calculus. If we sum over the internal vertices and add the leaf terms, then all terms involving $Q_u$ with $u\ne o$ cancel. Thus
\begin{equation}\label{eq:treeenergy}
 \norm{X_\flat}_2^2
 =\norm{Q_oXQ_o}_2^2+\sum_{v\in\mathcal T^\circ}\norm{p_vXp_v}_2^2.
\end{equation}
For an internal vertex, the spectrum of $p_vXp_v$ on $p_v\cH$ lies in $(\kappa,2\kappa]$. Since $\lambda^2+2\kappa^2\le3\kappa\lambda$ on this interval,
\[
 \norm{p_vXp_v}_2^2+2\kappa^2\rank(p_v)
 \le3\kappa\Tr(p_vXp_v).
\]
For a leaf, $2\kappa^2\rank(p_\lambda)\le2\kappa\Tr(p_\lambda Xp_\lambda)$, while $\norm{Q_oXQ_o}_2^2\le\kappa\Tr(Q_oXQ_o)$. Adding these inequalities and using the orthogonal decomposition $I=Q_o+\sum_up_u$ proves the second assertion. The first follows from entrywise multiplication.
\end{proof}

\subsection{The rank of $S_a(X-X_\flat)$}
For a strict ancestor $v\prec u$, let $s(v,u)$ denote the child of $v$ whose block is disjoint from $\Omega_u$. Set
\[
 U_u=p_uX(I-P_u)+\mathbf1_{\{u\in\mathcal L(\mathcal T)\}}p_uXp_u,
 \qquad V_u=\sum_{v\prec u}Q_{s(v,u)}Xp_u.
\]

\begin{lemma}\label{lem:discarded}
We have $X-X_\flat=\sum_{u\in\mathcal T}(U_u+V_u)$ and
\begin{equation}\label{eq:badrank}
 \rank\bigl(S_a(X-X_\flat)\bigr)
 \le2^{r+1}\sum_{u\in\mathcal T}\rank(p_u)
 \le\frac{2^{r+1}}\kappa\Tr(X).
\end{equation}
\end{lemma}

\begin{proof}
Decompose $X$ according to \eqref{eq:treeblocks}. At a leaf,
\[
 P_\lambda XP_\lambda-Q_\lambda XQ_\lambda=p_\lambda Xp_\lambda,
\]
because the mixed spectral terms vanish. On $\Omega_{v_-}\times\Omega_{v_+}$,
\[
 P_{v_-}XP_{v_+}-Q_{v_-}XQ_{v_+}
 =(P_{v_-}-Q_{v_-})XP_{v_+}+Q_{v_-}X(P_{v_+}-Q_{v_+}).
\]
Insert $P_w-Q_w=\sum_{w\preceq u}p_u$ and use the analogous identity on the opposite corner. For a fixed $u$, this yields $p_uXP_{s(v,u)}$ and $Q_{s(v,u)}Xp_u$ once for each $v\prec u$. The complementary child projections along the chain from $o$ to $u$ are orthogonal and satisfy
\[
 \sum_{v\prec u}P_{s(v,u)}=I-P_u.
\]
This proves the identity.

Lemma~\ref{lem:profiles} gives $\rank(S_a(U_u))\le2^r\rank(p_u)$, including the leaf case. Since $V_u=(I-P_u)V_up_u$, the same lemma gives $\rank(S_a(V_u))\le2^r\rank(p_u)$. Therefore, rank subadditivity and \eqref{eq:rankbudget} give \eqref{eq:badrank}.
\end{proof}

\begin{proof}[Proof of Theorem~\ref{thm:tree}]
First suppose $X\ge0$. For $t>0$, set $\kappa=t2^{r/2}/A$ and $b=\sum_u\rank(p_u)$. The counting estimates, Lemma~\ref{lem:discarded} and Lemma~\ref{lem:retained} give
\[
 \begin{aligned}
 N(t;S_a(X))
 &\le2^{r+1}b+\frac{A^2}{t^2}\norm{X_\flat}_2^2\\
 &=\frac{A^2}{t^2}\bigl(2\kappa^2b+\norm{X_\flat}_2^2\bigr)
 \le\frac{3A^2\kappa}{t^2}\Tr(X).
 \end{aligned}
\]
Consequently,
\begin{equation}\label{eq:positive-tree}
 \norm{S_a(X)}_{1,\infty}\le3\cdot2^{r/2}A\Tr(X).
\end{equation}

For general $X$, let $X=U|X|$ be its polar decomposition. We then have
\[
 \mathcal X=\begin{pmatrix}|X^*|&X\\X^*&|X|\end{pmatrix}
 =\begin{pmatrix}U\\I\end{pmatrix}|X|\begin{pmatrix}U^*&I\end{pmatrix}\ge0,
 \qquad\Tr(\mathcal X)=2\norm X_1.
\]
After reordering coordinates, this is a matrix on the same block decomposition with terminal sets $N_\lambda\times\{1,2\}$. The amplified symbol $a\otimes\mathbf J_2$ has the same complexity and supremum norm, and $S_a(X)$ is a corner of $S_{a\otimes\mathbf J_2}(\mathcal X)$. Compression decreases singular values, so \eqref{eq:positive-tree} proves $W(a)\le6\cdot2^{r/2}A$. As observed at the start of the section, the same bound holds for $\Wa(a)$.

To prove the lower bound, take the full binary tree in which every leaf has depth $r+1$, with singleton leaf sets. There are $n=2^{r+1}$ leaves, and every symbol on these leaves has complexity at most $r$, since $r$ refinements below either child of any split reach only leaves. Take
\[
 a=\begin{pmatrix}1&1\\1&-1\end{pmatrix}^{\otimes(r+1)}.
\]
Then $aa^*=nI$. Testing $\mathbf J_n/n$, a rank-one matrix of trace norm one, gives $W(a)\ge\sqrt n$. Factoring $a$ through its rows and the coordinate vectors gives $\gamma_2(a)\le\sqrt n$, hence $\Wa(a)\le\sqrt n$. Thus $W(a)=\Wa(a)=2^{(r+1)/2}$.
\end{proof}

\begin{proposition}[Rank across tree cuts]\label{prop:cutrank}
For a finite binary block decomposition of $\Omega$ indexed by $\mathcal T$, let $a=(a_{ij})_{i,j\in\Omega}$ be any scalar symbol constant on every terminal rectangle $N_\lambda\times N_\mu$. Suppose $L\ge1$ is an integer and, for every vertex $u$,
\begin{equation}\label{eq:cutrank}
 \begin{aligned}
 \rank\bigl(a|_{\Omega_u\times(\Omega\setminus\Omega_u)}\bigr)&\le L,\\
 \rank\bigl(a|_{(\Omega\setminus\Omega_u)\times\Omega_u}\bigr)&\le L.
 \end{aligned}
\end{equation}
Then
\begin{equation}\label{eq:cutrankweak}
 \Wa(a)\le6\sqrt L\,\maxnorm a.
\end{equation}
Zero coordinate projections are permitted in the tree.
\end{proposition}

\begin{proof}
Put $A=\maxnorm a$. If $A=0$, then the assertion is immediate, so assume $A>0$. Fix $X\ge0$ and $\kappa>0$, and define $p_u$ and $X_\flat$ by \eqref{eq:cuculescu} and \eqref{eq:retained}.
For each $u\in\mathcal T$, the first rank bound in \eqref{eq:cutrank} gives a factorization
\[
 a_{ij}=\sum_{\nu=1}^Lf_\nu(i)g_\nu(j),
 \qquad i\in\Omega_u,\quad j\notin\Omega_u,
\]
adding zero summands if necessary. Let $A_\nu$ and $B_\nu$ be the diagonal operators given by $f_\nu$ and $g_\nu$, extended by zero outside their respective blocks. For every $Y\in\cB(\cH)$ and every orthogonal projection $e\le P_u$,
\[
 S_a(eY(I-P_u))=
 \sum_{\nu=1}^L A_\nu eY(I-P_u)B_\nu,
\]
so its rank is at most $L\rank(e)$. The second rank bound in \eqref{eq:cutrank} gives the analogous estimate for $S_a((I-P_u)Ye)$ by the same argument. At a leaf, constancy on the terminal rectangles shows directly that
$S_a(eY(I-P_u)+eYe)$ has left support at most $e$, as in the last assertion of Lemma~\ref{lem:profiles}.

The decomposition of $X-X_\flat$ in the proof of Lemma~\ref{lem:discarded} is independent of the symbol. The preceding rank bounds therefore give
\[
 \rank\bigl(S_a(X-X_\flat)\bigr)
 \le2L\sum_u\rank(p_u).
\]
Combining this with Lemma~\ref{lem:retained}, and choosing
$\kappa=t\sqrt L/\maxnorm a$, gives
\[
 N(t;S_a(X))\le\frac{3\sqrt L\,\maxnorm a}{t}\Tr(X)
 \qquad(X\ge0).
\]
For arbitrary inputs, the positive $2$-by-$2$ dilation used in the proof of Theorem~\ref{thm:tree} gives the factor $2$. Moreover, under amplification, a cut matrix is replaced by its tensor product with $\mathbf J_d$, which has the same rank. Thus \eqref{eq:cutrankweak} follows. Zero projections contribute neither rank nor Hilbert--Schmidt energy.
\end{proof}

\section{The one-dimensional localization}\label{sec:grids}

We use two adjacent dyadic grids \cite{Mei03,HK12} and a $C^1$ partition of unity to localize away from the diagonal at each scale.

In this section and the next, let $M\in C^1(\RR^2\setminus\{x=y\})$ satisfy $H_{1,1}(M)<\infty$, with $K_0$ as in \eqref{eq:HMSnorm}. Put
\[
 K_1=\sup_{x\ne y}|x-y|
 \bigl(|\partial_xM(x,y)|+|\partial_yM(x,y)|\bigr),
 \qquad H_{1,1}(M)=K_0+K_1.
\]

For $\varepsilon\in\{+1,-1\}$ and $k\in\Z$, set
\begin{equation}\label{eq:grids}
 \begin{aligned}
 \theta_k^\varepsilon&=\varepsilon(-1)^k/3,\\
 \cD_k^\varepsilon&=\{2^k[\nu+\theta_k^\varepsilon,\nu+1+\theta_k^\varepsilon):\nu\in\Z\},
 \qquad \cD^\varepsilon=\bigcup_{k\in\Z}\cD_k^\varepsilon.
 \end{aligned}
\end{equation}
We abbreviate the signs $+1,-1$ to $+,-$ in subscripts and superscripts. For a bounded half-open interval $Q$, let $l_Q$ be its left endpoint and $|Q|$ its length, so that $Q=[l_Q,l_Q+|Q|)$. Write $Q_-,Q_+$ for its left and right halves, and put
\[
 \mathcal R(Q)=(Q_-\times Q_+)\cup(Q_+\times Q_-).
\]

\begin{lemma}\label{lem:grids}
The grids in \eqref{eq:grids} have the following properties.
\begin{enumerate}[label=\textup{(\roman*)},leftmargin=*]
\item Each interval at scale $2^{k+1}$ is the union of two intervals at scale $2^k$ in the same grid. In particular, intervals in one grid are nested or disjoint.
\item An endpoint from $\cD_k^+$ and an endpoint from $\cD_k^-$ are separated by at least $2^k/3$.
\item Every bounded subset of $\RR$ is contained in an interval of either prescribed grid.
\end{enumerate}
\end{lemma}

\begin{proof}
For (i), observe that $2\theta_{k+1}^\varepsilon-\theta_k^\varepsilon=\varepsilon(-1)^{k+1}\in\Z$. Thus a scale-$2^{k+1}$ interval has endpoints on the scale-$2^k$ lattice and length twice the mesh. For (ii), the union of the two endpoint lattices is $2^k(\Z+\{1/3,2/3\})$, whose successive gaps are $2^k/3$ and $2^{k+1}/3$. For (iii), the interval containing zero is either $2^k[-2/3,1/3)$ or $2^k[-1/3,2/3)$. Both contain $(-2^k/3,2^k/3)$, which exhausts $\RR$ as $k\to\infty$.
\end{proof}

Choose the $C^1$ cutoff
\[
 \chi(t)=\begin{cases}0,&t\le0,\\3t^2-2t^3,&0<t<1,\\1,&t\ge1,\end{cases}
 \qquad
 \phi(t)=\chi(12t)\chi(12(1-t)),\quad
 \psi(t)=\chi(24t-1)\chi(2-6t).
\]
Since $0\le\chi'\le3/2$, these functions take values in $[0,1]$ and satisfy
\begin{equation}\label{eq:cutoffs}
 \begin{aligned}
 &\phi=1\text{ on }[1/12,11/12],\quad
 \operatorname{supp}\phi\subset[0,1],\quad\norm{\phi'}_\infty\le18,\\
 &\psi=1\text{ on }[1/12,1/6],\quad
 \operatorname{supp}\psi\subset[1/24,1/3],\quad\norm{\psi'}_\infty\le36.
 \end{aligned}
\end{equation}
Since the derivatives of the two factors in each product have disjoint supports, we obtain the stated bounds.

Set $\phi_Q(x)=\phi((x-l_Q)/|Q|)$ and, for $x\ne y$, define
\begin{equation}\label{eq:weights}
 \begin{aligned}
 G_\varepsilon(x,y)&=\sum_{Q\in\cD^\varepsilon}
 \psi\Bigl(\frac{|x-y|}{|Q|}\Bigr)\phi_Q(x)\phi_Q(y),\\
 w_\varepsilon&=\frac{G_\varepsilon}{G_++G_-},\qquad F_\varepsilon=Mw_\varepsilon.
 \end{aligned}
\end{equation}
Set $F_\varepsilon(x,x)=0$.

\begin{lemma}\label{lem:partition}
Suppose $H_{1,1}(M)<\infty$ and put $K=24K_1+792K_0$.
\begin{enumerate}[label=\textup{(\roman*)},leftmargin=*]
\item $G_\pm$ are $C^1$ off the diagonal, $G_++G_-\ge1$, and
\begin{equation}\label{eq:weightest}
 0\le w_\varepsilon\le1,\qquad w_++w_-=1,\qquad
 |\partial_xw_\varepsilon|+|\partial_yw_\varepsilon|\le\frac{33}{|x-y|}.
\end{equation}
\item For every $Q\in\cD^\varepsilon$, throughout $\mathcal R(Q)$,
\begin{equation}\label{eq:shell}
 |F_\varepsilon|\le K_0,\qquad
 |\partial_xF_\varepsilon|+|\partial_yF_\varepsilon|\le\frac K{|Q|}.
\end{equation}
Moreover, $F_\varepsilon$ and its first derivatives vanish there when $|x-y|<|Q|/24$.
\end{enumerate}
\end{lemma}

\begin{proof}
Fix $x\ne y$ and put $\rho=|x-y|$.

First, a summand in $G_\varepsilon$, or one of its first derivatives, can be nonzero only if $3\rho<|Q|<24\rho$ and $x,y\in\operatorname{int}Q$. These conditions restrict the scale to finitely many values locally off the diagonal. Only finitely many intervals at each scale meet a fixed small neighborhood. Therefore the sum is locally finite and $C^1$. At a given point, at most one interval per scale contributes.

To prove $G_++G_-\ge1$, assume $x<y$ and choose a dyadic length $\ell$ with $6\rho\le\ell<12\rho$. The open interval
\[
 (x-\ell/12,y+\ell/12)
\]
has length at most $\ell/3$. By Lemma~\ref{lem:grids}(ii), it contains at most one endpoint from the union of the two scale-$\ell$ lattices. Therefore one of the grids has no endpoint in this interval, so in that grid there is a $Q$ with
\[
 [x-\ell/12,y+\ell/12)\subset Q.
\]
Thus $\phi_Q(x)=\phi_Q(y)=1$, and $\psi(\rho/\ell)=1$ because $1/12<\rho/\ell\le1/6$. That summand equals one, so $G_++G_-\ge1$.

Let $\ell$ be the smallest dyadic length greater than $3\rho$. The only possible contributing lengths are $\ell,2\ell,4\ell$, with a length at the boundary of the support contributing zero. At these three scales, respectively, $\rho/|Q|$ lies in $[1/6,1/3)$, $[1/12,1/6)$ and $[1/24,1/12)$. The corresponding bounds for $|\psi'|$ are $9,0,36$. For a summand at scale $|Q|$, the sum of its two absolute first derivatives is bounded by $(2|\psi'(\rho/|Q|)|+36)/|Q|$. Hence
\[
 |\partial_xG_\varepsilon|+|\partial_yG_\varepsilon|
 \le\frac{54}{\ell}+\frac{36}{2\ell}+\frac{108}{4\ell}
 =\frac{99}{\ell}\le\frac{33}{\rho}.
\]
The quotient rule and $G_++G_-\ge1$ give
\[
 |\partial_xw_+|+|\partial_yw_+|
 \le\frac{33}{\rho}\frac{G_-+G_+}{(G_++G_-)^2}
 \le\frac{33}{\rho}.
\]
The same bound holds for $w_-=1-w_+$, proving (i).

For (ii), take $(x,y)\in\mathcal R(Q)$. If an interval $Q'$ in the same grid contributes to $G_\varepsilon$ or its derivative, then $Q'$ contains both points. Since they lie in opposite halves of $Q$, nesting forces $Q\subseteq Q'$. We also have $|Q'|<24\rho$. Therefore, if $\rho<|Q|/24$, then no summand or derivative contributes, which proves the vanishing assertion. Otherwise the product rule gives
\[
 |\partial_xF_\varepsilon|+|\partial_yF_\varepsilon|
 \le\frac{K_1+33K_0}{\rho}
 \le\frac{24K_1+792K_0}{|Q|}=\frac K{|Q|}.
\]
Finally $|F_\varepsilon|\le|M|\le K_0$.
\end{proof}

\section{The one-dimensional model}\label{sec:trees}

Retain $F_\varepsilon$ from \eqref{eq:weights} and $K=24K_1+792K_0$ from Lemma~\ref{lem:partition}.

\subsection{The block decomposition of a finite sample}
Fix a grid $\cD^\varepsilon$ and a finite nonempty set $\Theta\subset\RR$. We construct a finite rooted full binary tree $\mathcal T$ whose vertices are subsets of $\Theta$. The root is $\Theta$, and singletons are leaves. For a subset $\mathcal C$ with at least two points, let $Q(\mathcal C)$ be the smallest interval of $\cD^\varepsilon$ containing it. Indeed, Lemma~\ref{lem:grids}(iii) gives a containing interval. Since all containing intervals form a chain and their dyadic lengths are bounded below by $\operatorname{diam}\mathcal C>0$, a smallest interval exists. The two children are
\[
 \mathcal C\cap Q(\mathcal C)_-,\qquad
 \mathcal C\cap Q(\mathcal C)_+.
\]
By minimality, both children are nonempty. Since each child has fewer points than its parent, the construction terminates.

Write $\mathcal C(v)$ for the subset at $v$ and $Q(v)=Q(\mathcal C(v))$ for an internal vertex. Thus we omit scales at which only one dyadic child meets the sample. We call this a \emph{compressed dyadic tree}. If $u$ lies $s\ge1$ edges below an internal vertex $v$, successive halving shows that $\mathcal C(u)$ lies in a dyadic interval of length at most
\begin{equation}\label{eq:clusterdiam}
 2^{-s}|Q(v)|.
\end{equation}

For each vertex, choose the midpoint of the extreme sample points,
\[
 \xi_u=\tfrac12\bigl(\min\mathcal C(u)+\max\mathcal C(u)\bigr).
\]
If $u'$ is a descendant of $u$, then $|\xi_{u'}-\xi_u|\le\tfrac12\operatorname{diam}\mathcal C(u)$. For $r\in\mathbb N_0$ and a leaf $\lambda$ below an internal $v$, let $v_\pm$ be the child containing it and let $u_r(\lambda;v)$ be the unique member of $\mathcal P_r(v_\pm)$ that is an ancestor of $\lambda$.

For a finite nonempty list $(x_i)_{i=1}^N$, put $\Omega=\{1,\ldots,N\}$ and take $\Theta$ to be its set of distinct values. Define $\Omega_u=\{i\in\Omega:x_i\in\mathcal C(u)\}$ and $N_\lambda=\Omega_\lambda$. If $i,i'$ belong to distinct leaves $\lambda,\lambda'$ with least common ancestor $v$, define
\begin{equation}\label{eq:freezing}
 h_{\varepsilon,r}(i,i')=
 F_\varepsilon\bigl(\xi_{u_r(\lambda;v)},\xi_{u_r(\lambda';v)}\bigr).
\end{equation}
Set $h_{\varepsilon,r}=0$ on the leaf blocks. If $\Theta$ is a singleton, all these matrices are zero.

\begin{lemma}\label{lem:expansion}
The symbol $h_{\varepsilon,r}$ has descendant complexity at most $r$, and $\maxnorm{h_{\varepsilon,0}}\le K_0$. For $r\ge1$, the difference
\[
 \delta_{\varepsilon,r}=h_{\varepsilon,r}-h_{\varepsilon,r-1}
\]
has complexity at most $r$ and
\[
 \maxnorm{\delta_{\varepsilon,r}}\le\min\{2K_0,K2^{-r-1}\}.
\]
Moreover,
\begin{equation}\label{eq:expansion}
 \bigl(F_\varepsilon(x_i,x_{i'})\bigr)_{i,i'=1}^N
 =h_{\varepsilon,0}+\sum_{r\ge1}\delta_{\varepsilon,r},
\end{equation}
with only finitely many nonzero summands.
\end{lemma}

\begin{proof}
For $u\in\mathcal P_r(v_-)$ and $u'\in\mathcal P_r(v_+)$, all leaf pairs below $u,u'$ have least common ancestor $v$. Therefore $h_{\varepsilon,r}$ is constant on $\Omega_u\times\Omega_{u'}$, with value $F_\varepsilon(\xi_u,\xi_{u'})$, and similarly on the transposed rectangle. For $r\ge1$, each $\Omega_u$ with $u\in\mathcal P_r(v_\pm)$ is contained in some $\Omega_w$ with $w\in\mathcal P_{r-1}(v_\pm)$. Hence $h_{\varepsilon,r-1}$ is constant on the same rectangles. All symbols vanish on the leaf blocks.

The bound for $h_{\varepsilon,0}$ and the bound $2K_0$ for each difference follow from $|F_\varepsilon|\le K_0$. Fix $r\ge1$ and indices $i\in N_\lambda$, $i'\in N_{\lambda'}$, where $\lambda,\lambda'$ are distinct leaves with least common ancestor $v$. Put $u=u_{r-1}(\lambda;v)$. If $u$ is a leaf, then the representatives at levels $r$ and $r-1$ coincide. Otherwise, $u$ lies exactly $r$ edges below $v$. The midpoint choice and \eqref{eq:clusterdiam} therefore bound their distance by $\tfrac12\operatorname{diam}\mathcal C(u)\le2^{-r-1}|Q(v)|$. The same is true in the other coordinate.

Let $\Delta x$ and $\Delta y$ denote the coordinate differences between the representative pairs at levels $r$ and $r-1$. Since both pairs lie in the same convex rectangle $Q(v)_-\times Q(v)_+$, or its reverse, we may integrate the gradient along the segment between them. Thus \eqref{eq:shell} gives
\[
 |\delta_{\varepsilon,r}(i,i')|
 \le \frac K{|Q(v)|}
 \max\bigl\{|\Delta x|,|\Delta y|\bigr\}\le K2^{-r-1}.
\]
For $r$ at least the depth of the tree, the vertices in \eqref{eq:freezing} are leaves, so $h_{\varepsilon,r}(i,i')=F_\varepsilon(x_i,x_{i'})$. Thus the differences eventually vanish, and summing them gives \eqref{eq:expansion}.
\end{proof}

\subsection{The atomic endpoint estimate}

\begin{proof}[Proof of Theorem~\ref{thm:HMSn} for $n=s=1$]
If $K_0=0$, then $M=0$ and there is nothing to prove. Suppose $K_0>0$ and fix a finite list $x_1,\ldots,x_N\in\RR$. For each $\varepsilon\in\{+1,-1\}$, put $f_\varepsilon=(F_\varepsilon(x_i,x_{i'}))_{i,i'}$. Theorem~\ref{thm:tree} and Lemma~\ref{lem:expansion} imply
\[
 \Wa(h_{\varepsilon,0})\le6K_0,\qquad
 \Wa(\delta_{\varepsilon,r})
 \le\min\{12K_0\,2^{r/2},3K\,2^{-r/2}\}.
\]
By Lemma~\ref{lem:weaksum}, applied to \eqref{eq:expansion} under each amplification, we have
\begin{equation}\label{eq:onegrid}
 \Wa(f_\varepsilon)
 \le\left(\sqrt{6K_0}+\sum_{r\ge1}
 \min\{\sqrt{12K_0}\,2^{r/4},\sqrt{3K}\,2^{-r/4}\}\right)^2.
\end{equation}
Choose the integer $R\ge0$ so that $2^R\le K/(4K_0)<2^{R+1}$. Sum the first bound for $r\le R$ and the second for $r>R$. The two geometric sums satisfy
\[
 \begin{aligned}
 \sqrt{12K_0}\sum_{r=1}^R2^{r/4}
 &\le\frac{\sqrt{12K_0}\,2^{(R+1)/4}}{2^{1/4}-1}
 \le C(K_0K)^{1/4},\\
 \sqrt{3K}\sum_{r>R}2^{-r/4}
 &=\frac{\sqrt{3K}\,2^{-R/4}}{2^{1/4}-1}
 \le C(K_0K)^{1/4}.
 \end{aligned}
\]
Since $K\ge K_0$ and $K\le792(K_0+K_1)$, \eqref{eq:onegrid} yields
\[
 \Wa(f_\varepsilon)\le C\sqrt{K_0K}
 \le C\sqrt{K_0(K_0+K_1)}.
\]

Since $w_++w_-=1$ off the diagonal and all symbols vanish on equal-value blocks, we have
\[
 (M(x_i,x_{i'}))_{i,i'}=f_++f_-.
\]
A second application of Lemma~\ref{lem:weaksum} proves \eqref{eq:endpointn} for square samples, after increasing the absolute constant.

Rectangular samples are corners of square samples on the combined row and column lists, so Remark~\ref{rem:rectangular} gives the same bound.

\end{proof}

\subsection{Passage to Hilbert--Schmidt kernels}

\begin{proof}[Proof of Corollary~\ref{cor:continuousn} for $n=s=1$]
Define the finite-range maps
\[
 \pi_n(x)=\begin{cases}2^{-n}\lfloor2^nx\rfloor,&x\in[-n,n),\\0,&x\notin[-n,n),\end{cases}
 \qquad M_n(x,y)=M(\pi_n(x),\pi_n(y)).
\]
For fixed $x\ne y$, we have $\pi_n(x)\ne\pi_n(y)$ for all sufficiently large $n$, and $(\pi_n(x),\pi_n(y))\to(x,y)$. Hence $M_n\to M$ almost everywhere and $|M_n|\le K_0$.

Fix $d\ge1$ and $X\in\Sone(L_2(\RR;\C^d))$. For each level set $E$ of $\pi_n$, choose an orthonormal basis of $L_2(E;\C^d)$, omitting zero spaces. Their union is an orthonormal basis of $L_2(\RR;\C^d)$. In this basis, $S_{M_n}\otimes\mathrm{id}_{M_d}$ is a Schur multiplier with entries $M(\pi_n(E),\pi_n(E'))$ between the subspaces associated with $E$ and $E'$. Here $\pi_n(E)$ denotes the value of $\pi_n$ on $E$. Since every finite corner is a restriction to a finite list, with repetitions allowed, the atomic estimate just proved and Lemma~\ref{lem:extension} give
\[
 \norm{(S_{M_n}\otimes\mathrm{id}_{M_d})X}_{1,\infty}
 \le C\sqrt{K_0(K_0+K_1)}\norm X_1.
\]
This extension agrees with multiplication of kernels by $M_n$, since both maps are bounded on $\Stwo$ and agree on the matrix units.

Let $k_X$ be the kernel of $X$. Then $k_X$ belongs to $L_2(\RR^2;M_d)$ with the Hilbert--Schmidt matrix norm. Since $|M_n-M|\le2K_0$, dominated convergence gives $\norm{(M_n-M)k_X}_{L_2}\to0$. Thus $(S_{M_n}\otimes\mathrm{id}_{M_d})X\to(S_M\otimes\mathrm{id}_{M_d})X$ in $\Stwo$, hence in operator norm. By continuity of singular values,
\[
 k\,s_k((S_M\otimes\mathrm{id}_{M_d})X)
 \le C\sqrt{K_0(K_0+K_1)}\norm X_1
 \qquad(k\ge1).
\]
Taking the supremum over $k$ and then over $d$ proves the assertion.
\end{proof}

\section{Product grids and smooth localization}\label{sec:productgrids}

For this section and the next, fix integers $n,s\ge1$ and $M\in C^s(\RR^{2n}\setminus\{x=y\})$ with $H_{n,s}(M)<\infty$, using \eqref{eq:HMSnorm}.

For $\varepsilon=(\varepsilon_1,\ldots,\varepsilon_n)\in\{-1,1\}^n$ and $k\in\Z$, set
\begin{equation}\label{eq:nproductgrids}
 \theta_{k,j}^{\varepsilon}=\varepsilon_j(-1)^k/3,
 \qquad
 \cD_k^\varepsilon=
 \{2^k(\nu+\theta_k^\varepsilon+[0,1)^n):\nu\in\Z^n\},
 \qquad \cD^\varepsilon=\bigcup_{k\in\Z}\cD_k^\varepsilon.
\end{equation}
Here $\theta_k^\varepsilon=(\theta_{k,1}^\varepsilon,\ldots,\theta_{k,n}^\varepsilon)$. These are $2^n$ nested dyadic cube grids. Indeed,
$2\theta_{k+1,j}^{\varepsilon}-\theta_{k,j}^{\varepsilon}\in\Z$ in each coordinate. At a fixed scale, endpoints in the two one-dimensional grids are separated by at least one third of the side length. These product grids are a concrete form of the adjacent systems in \cite{Mei03,HK12}.

Define new cutoff functions by
\[
 \eta(t)=\begin{cases}e^{-1/t},&t>0,\\0,&t\le0,\end{cases}
 \qquad
 \chi(t)=\frac{\eta(t)}{\eta(t)+\eta(1-t)}.
\]
Thus $\chi$ is smooth, equals $0$ on $(-\infty,0]$, and equals $1$ on $[1,\infty)$. Put
\[
 \phi(t)=\chi(12t)\chi(12(1-t)),
 \qquad
 \psi(t)=\chi(24t-1)\chi(2-6t).
\]
Then $\phi=1$ on $[1/12,11/12]$ and is supported in $[0,1]$, while $\psi=1$ on $[1/12,1/6]$ and is supported in $[1/24,1/3]$.

For a cube $Q=l_Q+[0,\ell(Q))^n$ of side length $\ell(Q)$, define
\[
 \phi_Q(x)=\prod_{j=1}^n
 \phi\left(\frac{x_j-l_{Q,j}}{\ell(Q)}\right).
\]
Writing $\rho=|x-y|$, set off the diagonal
\begin{equation}\label{eq:nweights}
 \begin{aligned}
 G_\varepsilon(x,y)&=\sum_{Q\in\cD^\varepsilon}
 \psi\left(\frac{\rho}{\ell(Q)}\right)\phi_Q(x)\phi_Q(y),\\
 G(x,y)&=\sum_{\varepsilon\in\{-1,1\}^n}G_\varepsilon(x,y),
 \qquad w_\varepsilon=G_\varepsilon/G,
 \qquad F_\varepsilon=Mw_\varepsilon.
 \end{aligned}
\end{equation}
We set $F_\varepsilon(x,x)=0$.

\begin{lemma}\label{lem:nlocalization}
The functions $G_\varepsilon$, $G$ and $w_\varepsilon$ in \eqref{eq:nweights} are smooth off the diagonal and satisfy
\[
 1\le G\le3\cdot2^n,
 \qquad 0\le w_\varepsilon\le1,
 \qquad \sum_\varepsilon w_\varepsilon=1,
\]
and, for every multi-index $\alpha$,
\begin{equation}\label{eq:nweightderivatives}
 |\partial_x^\alpha w_\varepsilon(x,y)|+
 |\partial_y^\alpha w_\varepsilon(x,y)|
 \le C_{n,\alpha}\rho^{-|\alpha|}.
\end{equation}

Subdivide $Q\in\cD^\varepsilon$ into its $2^n$ dyadic children by splitting successively in coordinates $1,\ldots,n$. At any intermediate binary split, let $R_-$ and $R_+$ be the two child rectangles. On $R_-\times R_+$ and on $R_+\times R_-$,
\begin{equation}\label{eq:nshell}
 |\partial_x^\alpha F_\varepsilon(x,y)|+
 |\partial_y^\alpha F_\varepsilon(x,y)|
 \le C_{n,s}H_{n,s}(M)\ell(Q)^{-|\alpha|},
 \qquad1\le|\alpha|\le s.
\end{equation}
On these products, $F_\varepsilon$ and the derivatives in \eqref{eq:nshell} vanish when $\rho<\ell(Q)/24$.
\end{lemma}

\begin{proof}
A summand of $G_\varepsilon$, or any of its derivatives, can be nonzero only when
\[
 3\rho<\ell(Q)<24\rho,
 \qquad x,y\in\operatorname{int}Q.
\]
At a fixed point there are at most three possible dyadic scales and at most one cube at each scale in a prescribed grid. Thus $G_\varepsilon$ is locally a finite smooth sum and $0\le G_\varepsilon\le3$.

Choose a dyadic length $\ell$ with $6\rho\le\ell<12\rho$. For each coordinate $j$, the interval
\[
 (\min\{x_j,y_j\}-\ell/12,
   \max\{x_j,y_j\}+\ell/12)
\]
has length at most $\ell/3$. Hence one of the two scale-$\ell$ endpoint lattices has no endpoint in this interval. Choose the corresponding sign $\varepsilon_j$. The product cube in $\cD^\varepsilon$ then contains $x$ and $y$ with distance at least $\ell/12$ from every face. Its summand in $G_\varepsilon$ equals one. Therefore $G\ge1$.

On a contributing scale, $\ell(Q)$ is comparable to $\rho$. Differentiation therefore gives
\[
 |\partial_x^\alpha G_\varepsilon|+|\partial_y^\alpha G_\varepsilon|
 \le C_{n,\alpha}\rho^{-|\alpha|},
\]
and the same estimate holds for $G$. Differentiating $G\,G^{-1}=1$ inductively, separately in the $x$ and $y$ variables, we obtain the corresponding estimate for $G^{-1}$. The product rule then gives \eqref{eq:nweightderivatives}.

Take $x,y$ in opposite rectangles at an intermediate split of $Q$. They lie in different full dyadic children of $Q$. If a cube $Q'$ in the same grid contributes at $(x,y)$, then it contains both points. By nesting, it therefore contains $Q$, and we also have $\ell(Q')<24\rho$. If $\rho<\ell(Q)/24$, then no term contributes. Otherwise, the product rule, \eqref{eq:HMSnorm}, and \eqref{eq:nweightderivatives} give
\[
 |\partial_x^\alpha F_\varepsilon|+
 |\partial_y^\alpha F_\varepsilon|
 \le C_{n,s}H_{n,s}(M)\rho^{-|\alpha|},
\]
which implies \eqref{eq:nshell}.
\end{proof}

\section{Polynomial approximation and the higher-dimensional endpoint}\label{sec:hdHMS}

Retain $n,s,M$ and the localized symbols $F_\varepsilon$ from Section~\ref{sec:productgrids}.

\subsection{Tensor interpolation with separate derivatives}
Fix $s$ distinct points in $(0,1)$. For a nondegenerate compact interval $J$, let $I_J$ denote Lagrange interpolation of degree at most $s-1$ at their affine images in $J$. There is a constant $\Lambda_s\ge1$, independent of $J$, such that, for every $f\in C(J)$,
\begin{equation}\label{eq:interpstable}
 \norm{I_Jf}_{L^\infty(J)}\le\Lambda_s\norm f_{L^\infty(J)}.
\end{equation}
If $J$ consists of one point, $I_J$ means evaluation there. Tensor interpolation on a box $B\subset\RR^n$ is denoted by $I_B$, with the same convention in degenerate coordinates. Superscripts $x$ and $y$ indicate the variable in which interpolation acts. Its range lies in
\[
 \mathcal Q_{s-1}(\RR^n)
 =\operatorname{span}\{x^\alpha:0\le\alpha_j\le s-1\},
 \qquad \dim\mathcal Q_{s-1}(\RR^n)=s^n.
\]

\begin{lemma}\label{lem:tensorapprox}
Let $B,C\subset\RR^n$ be compact boxes whose side lengths are at most $h$. If $F$ is $C^s$ on a neighborhood of $B\times C$, then
\begin{align}
 \norm{I_B^xI_C^yF}_\infty
 &\le\Lambda_s^{2n}\norm F_\infty,
 \label{eq:tensorstable}\\
 \norm{F-I_B^xI_C^yF}_\infty
 &\le C_{n,s}h^s\sum_{j=1}^n
 \left(\norm{\partial_{x_j}^sF}_\infty+
       \norm{\partial_{y_j}^sF}_\infty\right).
 \label{eq:tensorerror}
\end{align}
All norms in \eqref{eq:tensorstable}--\eqref{eq:tensorerror} are supremum norms on $B\times C$. 
\end{lemma}

\begin{proof}
The stability bound in one dimension follows from the fixed Lagrange basis on $[0,1]$ and rescaling. If $P$ is the Taylor polynomial of $f$ of degree $s-1$ at a point of $J$, then $I_JP=P$ and Taylor's formula gives
\[
 \norm{f-I_Jf}_\infty
 \le(1+\Lambda_s)\norm{f-P}_\infty
 \le C_s|J|^s\norm{f^{(s)}}_\infty.
\]

Let $I_1,\ldots,I_{2n}$ be the interpolation operators in the scalar coordinates of $B\times C$. These operators commute and have norm at most $\Lambda_s$. We have
\[
 I-I_1\cdots I_{2n}
 =\sum_{j=1}^{2n}I_1\cdots I_{j-1}(I-I_j).
\]
Applying the one-dimensional remainder estimate to $(I-I_j)F$ before the preceding bounded operators, we obtain \eqref{eq:tensorerror}. Multiplying the stability bounds gives \eqref{eq:tensorstable}.
\end{proof}

\subsection{Polynomial symbols on a finite sample}
Fix $\varepsilon\in\{-1,1\}^n$, a finite nonempty index set $\Omega$, and points $(x_i)_{i\in\Omega}$ in $\RR^n$. Choose a cube of $\cD^\varepsilon$ containing the list. Subdivide it by cycling through coordinates $1,\ldots,n$, and continue until the terminal dyadic cubes contain at most one distinct sample value each. Keep the resulting finite rooted full binary geometric tree $\mathcal T$. Empty vertices correspond to the zero projection and may be retained.

Let $\Omega_u$ be the indices whose points lie in the half-open rectangle at $u$, and let $P_u$ be the coordinate projection onto $\ell^2(\Omega_u)$. Repeated values remain in one terminal block. For every nonempty vertex, let
\begin{equation}\label{eq:samplebox}
 B_u=\prod_{j=1}^n
 \left[\min_{i\in\Omega_u}(x_i)_j,
       \max_{i\in\Omega_u}(x_i)_j\right].
\end{equation}
Then $B_u$ lies in its geometric rectangle, descendant boxes are nested, and every nonempty terminal vertex has a single-point box.

If $v$ is internal, let $Q(v)$ be the full dyadic cube being subdivided during the coordinate cycle containing the split at $v$. Points in opposite binary child rectangles lie in different full dyadic children of $Q(v)$. For $r\in\mathbb N_0$, every sample box $B_u$ reached after $nr$ further binary steps below either child, stopping at terminal vertices, has side lengths at most
\begin{equation}\label{eq:nshrink}
 2^{-r}\ell(Q(v)).
\end{equation}

For $i,j$ in distinct terminal blocks, let $v$ be the least common ancestor of their terminal vertices. Let $u_r(i;v)$ be the member of the partition after $nr$ refinements below the child of $v$ containing $i$, or the terminal vertex if it is reached earlier. Define
\begin{equation}\label{eq:ninterpolants}
 h_{\varepsilon,r}(i,j)=
 \left(I_{B_{u_r(i;v)}}^xI_{B_{u_r(j;v)}}^yF_\varepsilon\right)(x_i,x_j),
\end{equation}
and set $h_{\varepsilon,r}=0$ on the terminal diagonal blocks. The two interpolation boxes are compact subsets of opposite child rectangles. Since their finite sample coordinates have a positive gap in the splitting direction, $F_\varepsilon$ is $C^s$ on a neighborhood of their product.

\begin{lemma}\label{lem:nincrements}
Put $H=H_{n,s}(M)$ and
$f_\varepsilon=(F_\varepsilon(x_i,x_j))_{i,j\in\Omega}$. Then
\begin{equation}\label{eq:nerror}
 \maxnorm{h_{\varepsilon,r}}\le C_{n,s}K_0,
 \qquad
 \maxnorm{f_\varepsilon-h_{\varepsilon,r}}
 \le C_{n,s}H2^{-sr}.
\end{equation}
Set $a_{\varepsilon,0}=h_{\varepsilon,0}$ and
$a_{\varepsilon,r}=h_{\varepsilon,r}-h_{\varepsilon,r-1}$ for $r\ge1$. Then
\begin{equation}\label{eq:nincrementbound}
 f_\varepsilon=\sum_{r\ge0}a_{\varepsilon,r},
 \qquad
 \maxnorm{a_{\varepsilon,r}}
 \le C_{n,s}\min\{K_0,H2^{-sr}\},
\end{equation}
and the sum is finite. Each $a_{\varepsilon,r}$ is constant on every terminal rectangle. For every subtree cut, both scalar ranks of $a_{\varepsilon,r}$ in \eqref{eq:cutrank} are at most
\begin{equation}\label{eq:nrankbound}
 L_r=s^n2^{nr}.
\end{equation}
\end{lemma}

\begin{proof}
Equations \eqref{eq:tensorstable}, \eqref{eq:tensorerror}, \eqref{eq:nshell}, and \eqref{eq:nshrink} give \eqref{eq:nerror} on every off-diagonal terminal rectangle. Since both symbols vanish on terminal diagonal blocks, these estimates hold on the whole sample. Taking differences gives \eqref{eq:nincrementbound}; for $r=0$, we use $H\ge K_0$. If $nr$ exceeds the depth of the finite tree, then all interpolation boxes are singletons, and hence $h_{\varepsilon,r}=f_\varepsilon$.

All sample points in a terminal block coincide, so the increments are constant on terminal rectangles. It remains to prove the common cut-rank bound. Fix a vertex $u$. Recall that $\mathcal P_{nr}(u)$ indexes the partition after $nr$ refinements below $u$, retaining terminal vertices. Partition $\Omega_u$ according to $w\in\mathcal P_{nr}(u)$. For $j\notin\Omega_u$, the least common ancestor of $i\in\Omega_w$ and $j$ is a fixed strict ancestor $v$ of $u$, independent of $i$. The row interpolation cell used in \eqref{eq:ninterpolants} is an ancestor of $w$. Hence, as a function of $x_i$ on $\Omega_w$, the entries $h_{\varepsilon,r}(i,j)$ and, when $r\ge1$, $h_{\varepsilon,r-1}(i,j)$ belong to the same space $\mathcal Q_{s-1}(\RR^n)$, independently of the outside column $j$. Therefore there are scalar coefficients $b_{w,\alpha}(j)$ such that
\begin{equation}\label{eq:commonpolynomialrows}
 a_{\varepsilon,r}(i,j)=
 \sum_{\alpha\in\{0,\ldots,s-1\}^n}
 x_i^\alpha b_{w,\alpha}(j),
 \qquad i\in\Omega_w,\quad j\notin\Omega_u.
\end{equation}
There are at most $2^{nr}$ cells $w$ and $s^n$ monomials. Therefore the row rank across the cut is at most $s^n2^{nr}$. Interchanging rows and columns proves the other bound.
\end{proof}

\subsection{Summation and passage to continuous kernels}

\begin{proof}[Completion of the proof of Theorem~\ref{thm:HMSn}]
Assume $s>n/2$. If $K_0=0$, then the assertion is immediate. We may therefore assume $K_0>0$. Writing $H=H_{n,s}(M)$, we obtain from Proposition~\ref{prop:cutrank} and Lemma~\ref{lem:nincrements}
\[
 \Wa(a_{\varepsilon,r})
 \le C_{n,s}2^{nr/2}\min\{K_0,H2^{-sr}\}.
\]
Applying Lemma~\ref{lem:weaksum} to the finite expansion in \eqref{eq:nincrementbound}, uniformly at every matrix level, we obtain
\begin{equation}\label{eq:nweakseries}
 \Wa(f_\varepsilon)
 \le C_{n,s}\left(
 \sum_{r\ge0}2^{nr/4}
 \min\{K_0^{1/2},H^{1/2}2^{-sr/2}\}
 \right)^2.
\end{equation}
Choose the integer $R\ge0$ such that
$2^{sR}\le H/K_0<2^{s(R+1)}$. For $r\le R$, we use the first term in the minimum, and for $r>R$, we use the second. Since $s>n/2$, we have
\begin{align*}
 K_0^{1/2}\sum_{r=0}^R2^{nr/4}
 &\le C_nK_0^{1/2}2^{nR/4},\\
 H^{1/2}\sum_{r>R}2^{(n/4-s/2)r}
 &\le C_{n,s}K_0^{1/2}2^{nR/4}.
\end{align*}
Consequently,
\[
 \Wa(f_\varepsilon)
 \le C_{n,s}K_0\,2^{nR/2}
 \le C_{n,s}K_0^{\,1-n/(2s)}H^{\,n/(2s)}.
\]

Since $M=\sum_\varepsilon F_\varepsilon$ off the diagonal, another application of Lemma~\ref{lem:weaksum} over the $2^n$ product grids proves \eqref{eq:endpointn} for square samples. A rectangular sample is a corner of the square sample on the combined row and column lists. Thus Remark~\ref{rem:rectangular} completes the proof.
\end{proof}

\begin{proof}[Completion of the proof of Corollary~\ref{cor:continuousn}]
Choose finite-range measurable maps
$\pi_m:\RR^n\to\RR^n$ such that $\pi_m(x)\to x$ for every $x$. For example, round each coordinate to a mesh of size $2^{-m}$ on $[-m,m)^n$ and send the complement to zero. Put
\[
 M_m(x,y)=M(\pi_m(x),\pi_m(y)),
\]
with the zero convention when $\pi_m(x)=\pi_m(y)$. Then $|M_m|\le K_0$ and $M_m(x,y)\to M(x,y)$ whenever $x\ne y$.

Fix $d\ge1$ and $X\in\Sone(L_2(\RR^n;\C^d))$. Decompose $L_2(\RR^n;\C^d)$ according to the finitely many level sets of $\pi_m$ and choose an orthonormal basis in every nonzero summand. In the combined basis, $S_{M_m}\otimes\mathrm{id}_{M_d}$ is constant on the associated matrix blocks. Since every finite corner is an atomic sample covered by Theorem~\ref{thm:HMSn}, with repetitions allowed, Lemma~\ref{lem:extension} gives
\[
 \norm{(S_{M_m}\otimes\mathrm{id}_{M_d})X}_{1,\infty}
 \le C_{n,s}K_0^{\,1-n/(2s)}H_{n,s}(M)^{\,n/(2s)}\norm X_1.
\]
As in the one-dimensional case, this extension agrees with multiplication of Hilbert--Schmidt kernels by $M_m$.

Let $k_X$ be the matrix-valued Hilbert--Schmidt kernel of $X$. Since the diagonal has measure zero, dominated convergence gives
\[
 \int_{\RR^n\times\RR^n}|M_m(x,y)-M(x,y)|^2
 \norm{k_X(x,y)}_2^2\,dx\,dy\longrightarrow0.
\]
Thus the output operators converge in $\Stwo$, and hence in operator norm. By continuity of each singular value, the preceding estimate holds for $S_M\otimes\mathrm{id}_{M_d}$. Since the constant is independent of $d$, this proves the corollary.
\end{proof}

\section*{Acknowledgments and use of AI}
The author gratefully acknowledges assistance from OpenAI's GPT-6 Astra in obtaining many sharper constants and improving the exposition. The author takes full responsibility for the mathematical content and references.


\begin{thebibliography}{99}

\bibitem{BS87}
M. Sh. Birman and M. Z. Solomjak, \emph{Spectral theory of self-adjoint operators in Hilbert space}, Mathematics and its Applications (Soviet Series), D. Reidel Publishing Co., Dordrecht, 1987.

\bibitem{CGPT23}
J. M. Conde-Alonso, A. M. Gonz\'alez-P\'erez, J. Parcet, and E. Tablate, \emph{Schur multipliers in Schatten--von Neumann classes}, Ann. of Math. (2) \textbf{198} (2023), no. 3, 1229--1260. \href{https://doi.org/10.4007/annals.2023.198.3.5}{doi:10.4007/annals.2023.198.3.5}.

\bibitem{Cuc71}
I. Cuculescu, \emph{Martingales on von Neumann algebras}, J. Multivariate Anal. \textbf{1} (1971), no. 1, 17--27. \href{https://doi.org/10.1016/0047-259X(71)90027-3}{doi:10.1016/0047-259X(71)90027-3}.

\bibitem{Die17}
R. Diestel, \emph{Graph theory}, 5th ed., Graduate Texts in Mathematics, vol. 173, Springer, Berlin, 2017. \href{https://doi.org/10.1007/978-3-662-53622-3}{doi:10.1007/978-3-662-53622-3}.

\bibitem{GK69}
I. C. Gohberg and M. G. Kre\u{\i}n, \emph{Introduction to the theory of linear nonselfadjoint operators}, Translations of Mathematical Monographs, vol. 18, American Mathematical Society, Providence, RI, 1969.

\bibitem{GK70}
I. C. Gohberg and M. G. Kre\u{\i}n, \emph{Theory and applications of Volterra operators in Hilbert space}, Translations of Mathematical Monographs, vol. 24, American Mathematical Society, Providence, RI, 1970.

\bibitem{GPPW26}
A. M. Gonz\'alez-P\'erez, J. Parcet, J. P\'erez Garc\'ia, and S. Wang, \emph{Radial Schur multipliers on trees and the failure of Fourier--Schur transference}, preprint (2026), \href{https://arxiv.org/abs/2610.15881}{arXiv:2610.15881}.

\bibitem{HK12}
T. Hyt\"onen and A. Kairema, \emph{Systems of dyadic cubes in a doubling metric space}, Colloq. Math. \textbf{126} (2012), no. 1, 1--33. \href{https://doi.org/10.4064/cm126-1-1}{doi:10.4064/cm126-1-1}.

\bibitem{Mac61}
V. I. Macaev, \emph{Volterra operators obtained from self-adjoint operators by perturbation}, Dokl. Akad. Nauk SSSR \textbf{139} (1961), 810--813.

\bibitem{Mei03}
T. Mei, \emph{BMO is the intersection of two translates of dyadic BMO}, C. R. Math. Acad. Sci. Paris \textbf{336} (2003), no. 12, 1003--1006. \href{https://doi.org/10.1016/S1631-073X(03)00234-6}{doi:10.1016/S1631-073X(03)00234-6}.

\bibitem{Par26}
J. Parcet, \emph{The impact of Schur multipliers in harmonic analysis and operator algebras}, in \emph{International Congress of Mathematicians 2026}, vol. 4, SIAM, 2026, pp. 730--750. \href{https://doi.org/10.1137/25M1803371}{doi:10.1137/25M1803371}.

\bibitem{PPS89}
V. I. Paulsen, S. C. Power, and R. R. Smith, \emph{Schur products and matrix completions}, J. Funct. Anal. \textbf{85} (1989), no. 1, 151--178. \href{https://doi.org/10.1016/0022-1236(89)90050-5}{doi:10.1016/0022-1236(89)90050-5}.

\bibitem{Pin85}
A. Pinkus, \emph{$n$-Widths in approximation theory}, Ergebnisse der Mathematik und ihrer Grenzgebiete (3), vol. 7, Springer-Verlag, Berlin, 1985. \href{https://doi.org/10.1007/978-3-642-69894-1}{doi:10.1007/978-3-642-69894-1}.

\bibitem{Pis01}
G. Pisier, \emph{Similarity problems and completely bounded maps}, 2nd ed., Lecture Notes in Mathematics, vol. 1618, Springer-Verlag, Berlin, 2001.

\bibitem{Ran02}
N. Randrianantoanina, \emph{Non-commutative martingale transforms}, J. Funct. Anal. \textbf{194} (2002), no. 1, 181--212. \href{https://doi.org/10.1006/jfan.2002.3952}{doi:10.1006/jfan.2002.3952}.

\bibitem{Sio58}
M. Sion, \emph{On general minimax theorems}, Pacific J. Math. \textbf{8} (1958), no. 1, 171--176. \href{https://doi.org/10.2140/pjm.1958.8.171}{doi:10.2140/pjm.1958.8.171}.

\end{thebibliography}
\end{document}